\documentclass[11pt]{amsart}

\usepackage{amsrefs,amstext,amsmath,amsthm,amsfonts,amssymb,enumerate}

\usepackage{color}
\newtheorem{theorem}{Theorem}[section]
\newtheorem{proposition}[theorem]{Proposition}
\newtheorem{lemma}[theorem]{Lemma}
\newtheorem{corollary}[theorem]{Corollary}

\newtheorem{definition}[theorem]{Definition}
\newtheorem{remark}[theorem]{Remark}
\newtheorem{remarks}[theorem]{Remarks}

\newtheorem{questions}[theorem]{Questions}

\newcommand{\CC}{\ensuremath{\mathbb{C}}} 
\newcommand{\NN}{\ensuremath{\mathbb{N}}} 
\newcommand{\ZZ}{\ensuremath{\mathbb{Z}}} 

\newcommand{\lspan}{\mathop{\mathrm{span}}} 
\newcommand{\ran}{\mathop{\mathrm{ran}}} 
\newcommand{\supp}{\mathop{\mathrm{supp}}} 
\newcommand{\id}{\mathrm{id}} 
\newcommand{\defeq}{:=} 
\newcommand{\inv}{^{-1}} 

\newcommand{\HH}{\ensuremath{\mathcal{H}}} 
\newcommand{\ip}[2]{\left\langle #1 , #2 \right\rangle} 

\newcommand{\nrm}[1][\cdot]{\Vert #1 \Vert} 
\newcommand{\dualp}[2]{\left\langle #1 , #2 \right\rangle} 

\newcommand{\Bd}[1][\HH]{\ensuremath{\mathcal{B}(#1)}} 
\newcommand{\Cpt}[1][\HH]{\ensuremath{\mathcal{K}(#1)}} 
\newcommand{\btens}{\mathbin{\overline{\otimes}}} 

\newcommand{\Aut}{\mathop{\mathrm{Aut}}} 

\newcommand{\vN}[1][\grp]{\ensuremath{\mathrm{vN}(#1)}} 
\newcommand{\rgCst}[1][G]{\ensuremath{C^*_r(#1)}} 
\newcommand{\rcros}[3]{\ensuremath{#1 \rtimes_{{#3},r}  #2}} 
\newcommand{\rcrs}{\rcros{A}{G}{\alpha}} 
\newcommand{\EE}{\ensuremath{\mathcal{E}}} 

\newcommand{\falg}[1][\grp]{\ensuremath{A(#1)}} 
\newcommand{\Qcent}{Q_{\rm cent}} 

\newcommand{\CB}[1][A]{\ensuremath{\mathcal{CB}(#1)}} 
\newcommand{\cb}{{\mathrm{cb}}} 
\newcommand{\dec}{{\mathrm{dec}}} 

\newcommand{\Mcb}{{\mathrm{M}^{\mathrm{cb}}}} 
\newcommand{\HS}{{\mathrm{m}}} 
\newcommand{\Sch}{{\mathfrak{S}}} 
\newcommand{\HSmults}[3]{\ensuremath{\mathfrak{S}(#1,#2,#3)}} 
\newcommand{\HSm}{\HSmults{A}{G}{\alpha}} 

\begin{document}

\title{Exactness and SOAP of Crossed Products via Herz--Schur multipliers}

\author[A. McKee]{Andrew McKee}
\address{Department of Mathematical Sciences,
Chalmers University of Technology and  the University of Gothenburg,
Gothenburg SE-412 96, Sweden}
\email{amckee240@qub.ac.uk}

\author[L. Turowska]{Lyudmila Turowska}
\address{Department of Mathematical Sciences,
Chalmers University of Technology and  the University of Gothenburg,
Gothenburg SE-412 96, Sweden}
\email{turowska@chalmers.se}


\begin{abstract}
Given a $C^*$-dynamical system $(A,G,\alpha)$, with $G$ a discrete group, Schur $A$-multipliers and Herz--Schur $(A,G,\alpha)$-multipliers are used to implement approximation properties, namely exactness and the strong operator approximation property (SOAP), of $\rcrs$. The resulting characterisations of exactness and SOAP of $\rcrs$ generalise the corresponding statements for the reduced group $C^*$-algebra.
\end{abstract}

\maketitle

\section{Introduction}
\label{sec:intro}

\noindent
Recently in \cite{MTT18} the notion of classical Schur multipliers, which has been intensively studied in the literature (see {\it e.g.}\ \cites{Gro53,LaSa11,Pau02,Pis03}), was generalised to the operator-valued setting:
for a $C^*$-algebra $A\subset \Bd$ on a Hilbert space $\HH$, and a set $X$, Schur $A$-multipliers were defined as functions $\varphi :X \times X\to \CB[A,\Bd]$ such that the associated map $S_\varphi : \Cpt[\ell^2(X)] \otimes A\to \Cpt[\ell^2(X)] \otimes \Bd$ is completely bounded.  
Among many applications of classical Schur multipliers in Operator Theory is Ozawa's characterisation of exactness of the reduced $C^*$-algebra of a discrete group $G$ \cite{Oza00}: for a certain family of Schur multipliers $(\varphi_i : G \times G\to \CC)_i$ the associated maps $S_{\varphi_i}$ implement exactness of $C_r^*(G)$. 
In this paper we will use Schur $A$-multipliers to generalise this result to the setting of reduced crossed products. 

Schur multipliers are related to the notion of Herz--Schur multipliers associated to groups. 
The latter are functions $\psi : G \to \CC$ on a locally compact group $G$ that give rise to completely bounded maps on the reduced $C^*$-algebra. 
It is known that they constitute the ``invariant" part of the Schur multipliers on $G\times G$. 
Herz--Schur multipliers for the reduced crossed product were defined in \cite{BC16} for discrete groups and in \cite{MTT18} for general locally compact groups; in the latter they were also related to Schur $A$-multipliers in a similar manner as for the group case. 
Herz--Schur multipliers of special type are known to encode such approximation properties as nuclearity, the Haagerup approximation property, the completely bounded approximation property (CBAP) and the (strong) operator approximation property ((S)OAP) of the reduced $C^*$-algebra of a discrete group. 
Recently it was shown that Herz--Schur $(A,G,\alpha)$-multipliers do the same for the reduced crossed product $\rcrs$ and the first three approximation properties, see \cites{MSTT18,McKee}. 
B\'{e}dos--Conti~\cite{BC16}*{Section 4} have also used similar techniques to investigate regularity of $C^*$-dynamical systems: they give a condition involving Herz--Schur $(A,G,\alpha)$-multipliers which implies the full and reduced crossed products are canonically isomorphic.
In this note we shall give a characterisation of SOAP for $\rcrs$ in terms of Herz--Schur multipliers for crossed products.

The paper is organised as follows. In Section~\ref{sec:prelims} we fix notation and recall the notions of Schur $A$-multipliers and Herz--Schur multipliers of a $C^*$-dynamical system.
The main result of Section~\ref{sec:exactness} is a characterisation of exactness of a reduced crossed product in terms of the existence of certain Schur $A$-multipliers; from this characterisation we are able to deduce the known result that exactness is preserved by actions that satisfy a certain approximation property. 
Finally, in Section~\ref{sec:OAP} we give a characterisation of when a reduced crossed product has the strong operator approximation property in terms of the existence of certain Herz--Schur multipliers of the dynamical system. 
We finish with some observations about duality for the space of Herz--Schur multipliers and a certain commutative subspace thereof.

\section{Preliminaries}
\label{sec:prelims}

\noindent 
Throughout $G$ is a discrete group and $A$ is a unital $C^*$-algebra acting on the Hilbert space $\HH$.
Standard notation for tensor products will be used: the minimal tensor product of $C^*$-algebras will be written $A \otimes B$, and the normal spatial tensor product of von~Neumann algebras will be written $M \btens N$; the Hilbert space tensor product will be written $\HH \otimes \mathcal{K}$. 
We will make heavy use of the theory of operator spaces, complete boundedness and complete positivity; a suggested reference for background on these topics is \cite{EfRu00}.

\subsection{Dynamical Systems}
\label{ssec:dynamicalsystems}

Take a discrete group $G$ and a group homomorphism $\alpha : G \to \Aut(A)$; the triple $(A,G,\alpha)$ is called a $C^*$-dynamical system and $\alpha$ is called an action of $G$ on $A$. 
Using the representation $A \subseteq \Bd$ we define representations of $A$ and $G$ on $\ell^2(G) \otimes \HH \cong \ell^2(G,\HH)$ by
\[
	\big( \pi(a) \xi \big)(s) \defeq \alpha_s\inv(a) \xi(s) , \ \lambda_r \xi(s) \defeq \xi (r\inv s) , \ a \in A,\ r,s \in G,\ \xi \in \ell^2(G,\HH) .
\]
The pair $(\pi , \lambda)$ satisfies the covariance condition 
\[
	\lambda_r \pi(a) \lambda_r^* = \pi \big( \alpha_r(a) \big) , \quad a \in A ,\ r \in G ,
\]
and therefore defines a $*$-representation of $\ell^1(G,A)$ on $\ell^2(G,\HH)$ by its integrated form
\[
	\pi \rtimes \lambda (f) \defeq \sum_{s \in G} \pi \big( f(s) \big) \lambda_s , \quad f \in \ell^1(G,A) .
\]
The \emph{reduced crossed product} associated to the system $(A,G,\alpha)$ is the $C^*$-algebra defined as the closure of $\pi \rtimes \lambda(\ell^1(G,A))$ in the operator norm of $\Bd[\ell^2(G,\HH)]$, and denoted $\rcrs$.
We refer to \cite{BrO08}*{Section 4.1} for the details of this construction, including the fact that the resulting $C^*$-algebra $\rcrs$ does not depend on the initial choice of faithful representation $A \subseteq \Bd$.

Note that when $A = \CC$ and the action $\alpha$ is trivial the reduced crossed product $C^*$-algebra defined above is the reduced group $C^*$-algebra $\rgCst$.

Identify $\ell^2(G) \otimes \HH$ with $\oplus_{s \in G} \HH$, so that each element $x \in \Bd[\ell^2(G) \otimes \HH]$ can be identified in the usual way with a matrix $(x_{s,t})_{s,t \in G}$, where each entry comes from $\Bd[\HH]$.
One can check that the main diagonal of $x \in \rcrs$, that is $\{ x_{s,t} : s\inv t = e \}$, is given by $\{ \alpha_p\inv(a_e) : p \in G \}$ for some $a_e \in A$. 
The map $\EE_A$ which takes $x \in \rcrs$ to $a_e \in A$ is a conditional expectation \cite{BrO08}*{Proposition 4.1.9}, and is equivariant in the sense that
\[
	\alpha_s \big( \EE_A(x) \big) = \EE_A(\lambda_s x \lambda_s^*) , \quad s \in G .
\]

\subsection{Multipliers}
\label{ssec:multipliers}

In this section we summarise the definitions and results on Schur and Herz--Schur multipliers from \cite{MTT18} needed for this paper (though we use some conventions from \cite{MSTT18} which are more convenient).
Note that because we work only with discrete spaces and counting measure the separability assumptions of \cite{MTT18} are not required here.

Given $k \in \ell^2(X \times X ,A)$ define an operator $T_k : \ell^2(X,\HH) \to \ell^2(X,\HH) $ by
\[
 (T_k \xi)(x) \defeq \sum_{y \in X} k(x,y) \big( \xi(y) \big) , \quad  x \in X,\ \xi \in \ell^2(X , \HH) .
\]
It is easily checked that $T_k$ is a bounded operator and $\nrm[T_k] \leq \nrm[k]_2$. The collection of all such $T_k$ forms a dense subset of $\Cpt[\ell^2(X)] \otimes A$.
Let $\varphi : X \times X \to \CB[A , \Bd]$ be a bounded function and, for $k \in \ell^2(X \times X ,A)$, define
\[
	\varphi \cdot k (x,y) \defeq \varphi(x,y) \big( k(x,y) \big) , \quad x,y \in X .
\]
Clearly the map $S_\varphi : T_k \mapsto T_{\varphi \cdot k}$ is bounded with respect to the norm $\|\cdot\|_2$; we say $\varphi$ is a \emph{Schur $A$-multiplier} if $S_\varphi$ extends to a completely bounded map from $\Cpt[\ell^2(X)] \otimes A$ to $\Cpt[\ell^2(X)] \otimes \Bd$; in this case we write $\nrm[\varphi]_\Sch \defeq \nrm[S_\varphi]_\cb$.
Those functions $\varphi$ which are Schur $A$-multipliers are characterised by Stinespring-type dilation results --- see \cite{MTT18}*{Theorem 2.6} and \cite{MSTT18}*{Theorem 2.6} for a similar characterisation of when $S_\varphi$ is completely positive, in the latter case we say that $\varphi$ is a positive Schur-$A$-multiplier. 
We recall those results here.

\begin{theorem}\label{schura}
Let $X$ be a set, $\HH$ a Hilbert space, $A\subset \Bd$ a $C^*$-algebra, and $\varphi : X\times X \to \CB[A,\Bd]$ a bounded function. The following are equivalent:
\begin{enumerate}[i.]
	\item $\varphi$ is a Schur $A$-multiplier;
	\item there exists a Hilbert space $\HH_\rho$, a non-degenerate $*$-representation $\rho : A \to \Bd[\HH_\rho]$, and bounded functions $V, W : X \to \Bd[\HH,\HH_\rho]$ 
such that
\[
	\varphi(x,y)(a) = V(x)^* \rho(a) W(y), \quad x,y\in X,\ a\in A.
\]
\end{enumerate}
Moreover, if (i) holds true the functions $V$ and $W$ in (ii) can be chosen so that $\|\varphi\|_\Sch = \sup_{x\in X} \|V(x)\|\sup_{y\in X}\|W(y)\|$. 

If $\varphi$ in (i) is a positive Schur $A$-multiplier then one can choose $V=W$ in (ii) with the norm equality $\|\varphi\|_\Sch = \sup_{x\in X} \|\varphi(x,x) \|_\cb = \sup_{x\in X}\|V(x)\|^2$. 
\end{theorem}

Suppose that $\varphi$ is a Schur $A$-multiplier. Then the map $S_\varphi:\Cpt[\ell^2(X)] \otimes A\to\Cpt[\ell^2(X)] \otimes \Bd$ has a canonical extension to a map from $\Bd[\ell^2(X)]\otimes A$ into the von Neumann tensor product $\Bd[\ell^2(X)]\btens\Bd$ defined as follows.
Consider the second dual of the map $S_\varphi$: 
\[
    S_\varphi^{**} : ( \Cpt[\ell^2(G)] \otimes A)^{**} \to \Bd[\ell^2(G)] \btens \Bd^{**} .
\]
Writing $E$ for the conditional expectation from $\Bd[\ell^2(G)] \btens \Bd^{**}$ to the space $\Bd[\ell^2(G)] \btens \Bd$ we obtain a completely bounded map
\[
    \Phi \defeq E \circ S_\varphi^{**} : ( \Cpt[\ell^2(G)] \otimes A)^{**} =\Bd[\ell^2(X)]\btens A^{**}\to \Bd[\ell^2(G) \otimes \HH] ;
\]
we shall also write $S_\varphi$ for the restriction of $\Phi$ to $\Bd[\ell^2(X)]\otimes A$ and this is the required extension. It is easy to see that if $\rho$, $V$, $W$ are as in Theorem~\ref{schura}(ii), and considering $V$ and $W$ as operators from $\ell^2(X)\otimes \HH$ to $\ell^2(X)\otimes \HH_\rho$ we have that the extension satisfies
\[
	S_\varphi(T)=V^*(\id\otimes\rho)(T)W, \quad T\in \Bd[\ell^2(X)]\otimes A.
\]
In what follows writing $S_\varphi$ we shall often mean the extension to $\Bd[\ell^2(X)]\otimes A$.

Now consider a $C^*$-dynamical system $(A,G,\alpha)$ and let $F : G \to \CB$ be a bounded function. For $f \in \ell^1(G,A)$ define
\[
	F \cdot f (r) \defeq F(r) \big( f(r) \big) , \quad r \in G .
\]
We say that $F$ is a Herz--Schur $(A,G,\alpha)$-multiplier if the map $S_F : \pi \rtimes \lambda(f) \mapsto \pi \rtimes \lambda(F \cdot f)$ extends to a completely bounded map on $\rcrs$; in this case we write $\nrm[F]_\HS \defeq \nrm[S_F]_\cb$.
The collection of Herz--Schur $(A,G,\alpha)$-multipliers will be denoted $\HSm$.

The functions $F$ which are Herz--Schur $(A,G,\alpha)$-multipliers can be characterised in terms of Schur $A$-multipliers: given $F : G \to \CB$ define
\[
	\mathcal{N}(F) (s,t)(a) \defeq \alpha_{s\inv} \big( F(st\inv) \big( \alpha_s(a) \big) \big) , \quad s,t \in G,\ a \in A .
\]
Then $F$ is a Herz--Schur $(A,G,\alpha)$-multiplier if and only if $\mathcal{N}(F)$ is a Schur $A$-multiplier, and in this case $\nrm[F]_\HS = \nrm[\mathcal{N}(F)]_\Sch$ \cite{MTT18}*{Theorem 3.8}.

\section{Exactness}
\label{sec:exactness}

\noindent
Recall that a $C^*$-algebra $A$ is called exact if it has a faithful nuclear representation $\pi:A\to \Bd$, \textit{i.e.}\ there exists a net $(k_i)_{i\in I}$ of natural numbers and completely positive contractions $\varphi_i: A\to M_{k_i}(\mathbb C)$ and $\psi_i:M_{k_i}(\mathbb C)\to \Bd$ such that $\|\pi(x)-(\psi_i\circ\varphi_i)(x)\|\to 0$ for every $x\in A$. 
The notion is independent of the representation $\pi$: if $A\subset \Bd$ is a concretely represented exact $C^*$-algebra then the identity representation $\id:A\to \Bd$ is nuclear.  
It follows from Kirchberg's characterisation of exactness that one  can replace the existence of completely contractive positive maps in the definition of exactness by completely bounded maps with uniformly bounded cb norms (see \cite{BrO08}*{Proposition 3.7.8, Theorem 3.9.1}). 

Recall also that a bounded linear map $\phi : A \to B$ between two $C^*$-algebras is called decomposable if there exist two completely positive maps $\psi_i : A \to B$ ($i=1,2$) such that the map $\tilde{\phi} : A\to M_2(B)$ defined by
\[
	\tilde{\phi}(a) := \begin{pmatrix} \psi_1(a) & \phi(a) \\ \phi(a^*)^* & \psi_2(a) \end{pmatrix}
\]
is completely positive. In this case one defines 
\[
	\|\phi\|_\dec := \inf\{\max\{\|\psi_1\|, \|\psi_2\|\}\},
\] 
where the infimum is taken over all $\psi_1$, $\psi_2$ as above. 
We have $\|\phi\|_\cb\leq \|\phi\|_\dec$, and if $B$ is injective $\|\phi\|_\dec= \|\phi\|_\cb$ \cite{EfRu00}*{Lemma 5.4.3}.

The following result is inspired by Dong--Ruan~\cite{DR12}*{Theorem 6.1}. 
We note that a different definition of exactness of $C^*$-dynamical systems has been given by Sierakowski~\cite{Sie10} to study the ideal structure of crossed products; this definition was also studied by B\'{e}dos--Conti~\cite{BC15}*{page 63}, where Corollary~\ref{co:amenableaction} below was established with a different proof. 

\begin{definition}\label{de:exactsystem} Let $A\subset \Bd$. 
We say a $C^*$-dynamical system $(A,G,\alpha)$ is \emph{exact} if there exists a net $(\varphi_i)_i$ of Schur $A$-multipliers such that
\begin{enumerate}[i.]
    \item $\sup_i \nrm[\varphi_i]_\Sch < \infty$;
    \item for each $i$ there exists a finite set $F_i \subseteq G$ such that $\varphi_i$ is supported on $\Delta_{F_i} \defeq \{ (s,t) : st\inv \in F_i \}$;
    \item for all $a \in A$ we have $\nrm[\varphi_i(s,t)(\alpha_s\inv(a)) - \alpha_s\inv(a)] \stackrel{i}{\to} 0$ uniformly on $\Delta_R$ for each finite $R \subseteq G$;
    \item for each $i$ and each $r \in G$ the space $\{ t \mapsto (\varphi_i(t,r\inv t)(\alpha_t\inv(a))) : a \in A \}$ is a finite-dimensional subspace of $\ell^\infty(G , \Bd)$. 
\end{enumerate}
\end{definition}

\begin{remark}\label{re:exactSchmultspos}
{\rm
The proof of Theorem~\ref{th:exactcrossediffexactsystem} below shows that the Schur $A$-multipliers in Definition~\ref{de:exactsystem} may be chosen to be positive.

Note also that condition (iv) above implies $\varphi_i(s,t)$ is a finite rank map on $A$ for all $s,t \in G,\ i \in I$, but the latter is not equivalent to (iv).
}	
\end{remark}

\begin{theorem}\label{th:exactcrossediffexactsystem}
Let $(A,G,\alpha)$ be a $C^*$-dynamical system, with $G$ a discrete group. 
The following are equivalent:
\begin{enumerate}[i.]
    \item $(A,G,\alpha)$ is exact;
    \item $\rcrs$ is exact.
\end{enumerate}
\end{theorem}
\begin{proof}
(i) $\implies$ (ii) Let $S_i$ be the completely bounded map from $\Bd[\ell^2(G)] \otimes A$ to $\Bd[\ell^2(G)] \btens \Bd$ associated to $\varphi_i$.  We observe that 
$\rcrs\subset \Bd[\ell^2(G)] \otimes A$
and write $\Phi_i$ for the restriction of this map to $\rcrs$.

Each $\Phi_i$ is a finite rank map since it is supported on finitely many diagonals, and the image of each diagonal is finite-dimensional; here we think of operators $T$ on $\ell^2(G)\otimes\HH$ as block operator matrices $(a_{p,q})_{p,q\in G}$ where $\langle a_{p,q} \xi,\eta\rangle=\langle T\delta_q\otimes\xi,\delta_p\otimes\eta\rangle$.

Now, since $\Bd[\ell^2(G) \otimes \HH]$ is an injective $C^*$-algebra we have 
\[
    \nrm[\Phi_i]_\dec = \nrm[\Phi_i]_\cb \leq \nrm[\varphi_i]_\Sch < C 
\]
for some $C>0$. 
Since $\Phi_i$ is finite rank by \cite{Pis03}*{Theorem 12.7} there exist a natural number $n$ and completely bounded maps $\nu_i : \rcrs \to M_n$ and $\mu_i : M_n \to \Bd[\ell^2(G) \otimes \HH]$ such that $\Phi_i = \mu_i \circ \nu_i$ and
\[
    \nrm[\nu_i]_\cb \nrm[\mu_i]_\cb \leq \nrm[\Phi_i]_\dec \leq C .
\]

Since for $\xi$, $\eta\in\HH$, 
\[
    \ip{(\pi(a) \lambda_r)_{p,q} \xi}{\eta} = \ip{\pi(a) \lambda_r (\delta_q \otimes \xi)}{\delta_p \otimes \eta} =
        \begin{cases}
            0 & \text{if $p \neq rq$} \\
            \ip{\alpha_p\inv(a) \xi}{\eta} & \text{if $p = rq$,}
        \end{cases}
\]
we have
\[
\begin{split}
    \ip{\Phi_i(\pi(a) \lambda_r) (\delta_q \otimes \xi)}{\delta_p \otimes \eta} &= \ip{\varphi_i(p,q) \left( (\pi(a)\lambda_r)_{p,q} \right) \xi}{\eta} \\
        &= \begin{cases}
             0 & \text{if $p \neq rq$} \\
             \ip{\varphi_i(p,q)(\alpha_p\inv(a)) \xi}{\eta} & \text{if $p = rq$,}
           \end{cases}
\end{split}
\]
so conditions (i) and (iii) imply that $\Phi_i(\pi(a) \lambda_r)$ converges to $\pi(a) \lambda_r$ in norm. Since sums of elements of this form are dense in $\rcrs$ and the $S_i$, therefore the $\Phi_i$, are uniformly bounded it follows that $\nrm[\Phi_i(x) - x] \to 0$ for all $x \in \rcrs$.
This implies that $\rcrs$ is exact by the discussion at the beginning of this section. 

(ii) $\implies$ (i) By Ozawa \cite{Oza00}*{Lemma 2}, given a finite subset $B$ of $\rcrs$ and $\epsilon >0$ we may find $\theta : \rcrs \to \Bd[\ell^2(G) \otimes \HH]$ which is finite rank unital completely positive \ with $\nrm[\theta(x) - x] < \epsilon$ for all $x \in B$ of the form
\[
    \theta(x) = \sum_{k = 1}^d \omega_k(x) y_k , \quad x \in \rcrs ,
\]
where $y_k \in \Bd[\ell^2(G) \otimes \HH]$ and each $\omega_k$ is a vector functional of the form
\[
    \omega_k(y) = \ip{y (\delta_{p(k)} \otimes \xi_k)}{\delta_{q(k)} \otimes \eta_k} , \quad p(k) ,q(k) \in G ,\ \xi_k,\eta_k \in \HH .
\]
Now given such a $B$ and $\epsilon$ we define, for $s,t \in G ,\ a \in A$,
\[
    \varphi : G \times G \to \CB[A , \Bd] ;\ \varphi(s,t)(a) \defeq P_s \theta ( \pi(\alpha_s(a)) \lambda_{st\inv} ) P_t^* , 
\]
where $P_r : \ell^2(G) \otimes \HH \to \HH ;\ \sum_{s \in G} \delta_s \otimes \xi_s \mapsto \xi_r$.
To show that $\varphi$ is a Schur $A$-multiplier calculate:
\begin{equation}\label{eq:defisaposSch}
\begin{split}
    \varphi(s,t)(a) = P_s \theta ( \pi(\alpha_s(a)) \lambda_{st\inv} ) P_t^* &= P_s \theta (\lambda_s \pi(a) \lambda_t^*) P_t^* \\
        &= P_s V^* \rho (\lambda_s \pi(a) \lambda_t^*) V P_t^* \\
        &= ( P_s V^* \rho(\lambda_s) ) \rho(\pi(a)) ( \rho(\lambda_t)^* V P_t^* ) ,
\end{split}
\end{equation}
where $V^* \rho ( \cdot ) V$ is the Stinespring form of the completely positive map $\theta$.
It follows from \cite{MSTT18}*{Theorem 2.6} that $\varphi$ is a positive Schur $A$-multiplier.
Since 
\begin{equation}\label{V}
	\nrm[V]^2 = \nrm[\theta(1)] =1
\end{equation} 
it also follows that we have defined a net of positive Schur $A$-multipliers with uniformly bounded multiplier norm.

For the support condition observe that $\omega_k(\pi(a) \lambda_r) \neq 0$ only if $\delta_{q(k)} = \lambda_r \delta_{p(k)} = \delta_{r p(k)}$ for some $k$. Hence $\varphi(s,t)(a) \neq 0$ only if $st\inv \in \{ q(k) p(k)\inv : k = 1, \ldots , d \}.$

Given finite sets $F\subset A$, and $R\subset G$ we can construct $\varphi_{F,R,\epsilon}$ as above so that for all $(s,t)\in\Delta_R$ and $a\in F$
\[
\begin{split}
    \nrm[\varphi_{F,R,\epsilon}(s,t)(\alpha_s^{-1}(a)) - \alpha_s^{-1}(a)] &= \nrm[P_s \theta ( \pi(a) \lambda_{st\inv} ) P_t^* - P_s  \pi(a) \lambda_{st\inv} P_t^*] \\
        &\leq \nrm[\theta ( \pi(a) \lambda_{st\inv} ) - \pi(a) \lambda_{st\inv}] < \epsilon .
\end{split}
\]
Therefore the net $\{\varphi_{F,R,\epsilon}\}$, where $(F_1,R_1,\epsilon_1)\leq (F_2,R_2,\epsilon_2)$ iff
$F_1\subset F_2$, $R_1\subset R_2$, $\epsilon_2\leq\epsilon_1$, satisfies
\[
	\sup_{(s,t)\in\Delta_{\tilde R}}\nrm[\varphi_{F,R,\epsilon}(s,t)(\alpha_s^{-1}(a)) - \alpha_s^{-1}(a)]\to 0
\]
for any $a\in A$ and any finite $\tilde R\subset G$. 

Finally, the condition on the diagonals is satisfied because $\theta$ is finite rank:
\[
    \varphi_{F,R,\epsilon}(t,r\inv t)(\alpha_t^{-1}(a)) = \sum_{k =1}^d \omega_k \big( \pi(a)\lambda_r \big) P_t y_k P_{r\inv t}^*
\]
so $\{ t \mapsto \varphi_{F,R,\epsilon}(t,r\inv t)(\alpha_t\inv(a)): a \in A \}$ is a finite-dimensional subspace of $\ell^\infty(G , \Bd)$ for each $r\in G$.
\end{proof}

In previous work we observed that Herz--Schur multipliers can be used to prove that an amenable action on a nuclear $C^*$-algebra gives a nuclear crossed product \cite{MSTT18}*{Corollary 4.6}.
In the Corollary below we provide further evidence that Herz--Schur multipliers can be used as a technical basis for approximation results by using  Herz--Schur multipliers to prove the corresponding exactness result \cite{BrO08}*{Theorem 4.3.4 (3)} and its generalisation \cite{BC15}*{Theorem 5.8}. 

We recall from \cite{BC12}*{Definition 5.7} that $(A,G,\alpha)$ is said to have the approximation property, which we call \emph{Exel's approximation property} (\cites{exel, exelng}) if there exist nets $(\xi_i)_i$ and $(\eta_i)_i$ in $C_c(G,A)$ such that 
\begin{enumerate}
    \item $\sup_i\|\sum_{g\in G}\xi_i(g)^*\xi_i(g)\|\sup_i\|\sum_{g\in G}\eta_i(g)^*\eta_i(g)\|<\infty$;
    \item $\lim_i\|\sum_{h\in G}\xi_i(h)^*a\alpha_g(\eta_i(g^{-1}h))-a\|=0$, for all $a\in A$ and $g\in G$.
\end{enumerate}

We remark that if the action $\alpha : G \to \Aut(A)$ is amenable, see \cite{BrO08}*{Definition 4.3.1}, then $(A,G,\alpha)$ has Exel's approximation property with $\xi_i(r)=\eta_i(r)\in Z(A)^+$, for each $r\in G$, and $\sum_{g\in G}\xi_i(g)^2=1$ for all $i$. 
In the following result we give a multiplier-based proof of a part of \cite{BC15}*{Theorem 5.8}.


\begin{corollary}\label{co:amenableaction}
Let $(A,G,\alpha)$ be a $C^*$-dynamical system  with Exel's approximation property.
The reduced crossed product $\rcrs$ is exact if and only if $A$ is exact.
\end{corollary}
\begin{proof}
If $\rcrs$ is exact then the existence of a conditional expectation $\rcrs \to A$ immediately implies $A$ is exact.
Conversely, assume $A$ is exact, with $\Phi_j : A \to \Bd$ an approximating net of u.c.p.\ maps, and $\xi_i, \eta_i : G \to A$ nets implementing Exel's  approximation property.
We may assume that $A$ is faithfully embedded in $\HH$ such that the action is implemented by unitary operators, so for each $p \in G$ let $u_p$ be a unitary on $\HH$ such that $u_p a u_p^* = \alpha_p(a)$ ($a \in A$).  
Define
\[
    \varphi_{i,j}(s,t)(a) \defeq \sum_{p \in G} \alpha_s\inv \big( \xi_i(sp)^* \big) u_p \left( \Phi_j \big(\alpha_p\inv(a) \big) \right) u_p^* \alpha_t\inv \big( \eta_i(tp) \big) ,
\]
which gives a net of Schur $A$-multipliers with uniformly bounded multiplier norm as in \cite{MSTT18}*{Corollary 4.6}.

Since the $r$th term of the sum is non-zero only if $r \in s\inv F_i \cap t\inv G_i$ we have that $\varphi_{i,j}$ is supported on the strip $\Delta_{\mathcal{F}_i} \defeq \{ (s,t) : st\inv \in F_i G_i\inv \}$, where $F_i$ and $G_i$ are  finite supports  of $\xi_i$ and $\eta_i$ respectively. 
As $\Phi_j$ is finite rank it follows that $\varphi_{i,j}(s,t)$ is a finite rank map from $A$ to $\Bd$ for all $s,t$.

For each $t \in G$ and each $i,j$ the subspace of $\ell^\infty(G,\Bd)$ given by
\[
\begin{split}
    & \left\{ \big( \varphi_{i,j}(s,t\inv s) ( \alpha_s\inv(a)) \big)_s : a \in A \right\} \\
        &\quad = \left\{ \left( \sum_{p \in G} \alpha_s\inv \big( \xi_i(sp)^* \big) u_p \Phi_j \big(\alpha_{sp}\inv (a) \big) u_p^* \alpha_{s\inv t} \big( \eta_i(t\inv sp) \big) \right)_s : a \in A \right\} \\
        &\quad = \left\{ \left( \sum_{r \in G} \alpha_s\inv \big( \xi_i(r)^* \big) u_{s\inv r} \Phi_j \big(\alpha_{r}\inv (a) \big) u_{s\inv r}^* \alpha_{s\inv t} \big( \eta_i(t\inv r) \big) \right)_s : a \in A \right\}
\end{split}
\]
is finite-dimensional since $\Phi_j$ is finite rank and $\xi_i$, $\eta_i$  have finite supports.
Finally
\[
\begin{split}
    &\left\| \varphi_{i,j}(s,t) \big( \alpha_s\inv(a) \big) - \alpha_s\inv(a) \right\| \\
        &\qquad = \left\| \sum_{p \in G} \alpha_s\inv \big( \xi_i(sp)^* \big) u_p \Phi_j \big( \alpha_{sp}\inv(a) \big) u_p^* \alpha_t\inv \big(\eta_i(tp) \big) - \alpha_s\inv(a) \right\| \\
        &\qquad \leq \left\| \sum_{p \in G} \alpha_s\inv \big( \xi_i(sp)^* \big) u_p \Phi_j \big( \alpha_{sp}\inv(a) \big) u_p^* \alpha_t\inv \big( \eta_i(tp) \big) \right. \\ 
        &\qquad \qquad \qquad - \left. \sum_{p \in G} \alpha_s\inv \big( \xi_i(sp)^* \big) u_p \alpha_{sp}\inv(a) u_p^* \alpha_t\inv \big( \eta_i(tp) \big) \right\| \\
            &\qquad \qquad \qquad + \left\| \sum_{p \in G} \alpha_s\inv \big( \xi_i(sp)^* \big) \alpha_s\inv(a) \alpha_t\inv \big( \eta_i(tp) \big) - \alpha_s\inv(a) \right\| \\
        &\qquad \leq \left\| \sum_{p \in G} \alpha_s\inv \big( \xi_i(sp)^* \big) u_p \Big( \Phi_j \big( \alpha_{sp}\inv(a) \big) - \alpha_{sp}\inv(a) \Big) u_p^* \alpha_t\inv \big( \xi_i(tp) \big) \right\| \\
        &\qquad \qquad \qquad + \left\| \alpha_s\inv\big(\sum_{p \in G}  \xi_i(sp)^* a \alpha_{st\inv} \big( \eta_i(tp) \big) - a\big) \right\|  \\
        \end{split}\]
            \[
            \begin{split}
        & \qquad = \left\| \sum_{p \in G} \alpha_s\inv \big( \xi_i(sp)^* \big) u_p \Big( \Phi_j \big( \alpha_{sp}\inv(a) \big)- \alpha_{sp}\inv(a) \Big) u_p^* \alpha_t\inv \big( \eta_i(tp) \big) \right\| \\
        &\qquad \qquad \qquad + \left\| \sum_{p \in G} \xi_i(p)^*a \alpha_{st\inv } \big( \eta_i(ts\inv p) \big) - a\right\|  .
\end{split}
\]
The second term converges to 0 uniformly on $\Delta_R$ for any finite $R$ by the second condition for the  approximation property. 
If $st\inv$ belongs to the finite set $R$ the first term is non-zero only if $sp \in F_i \cap st\inv G_i \subseteq F_i \cap R G_i$, which is a finite set.
Given $\epsilon >0$ and $i$ one can find $j(i)$ such that $\nrm[\Phi_{j(i)}(\alpha_{sp}\inv(a)) - \alpha_{sp}\inv(a)] < \epsilon$ for all $sp \in F_i \cap R G_i$.
For $\xi$, $\eta\in \HH$ we have
\[
\begin{split}
	& \left| \ip{\sum_{p \in G} \alpha_s\inv (\xi_i(sp)^*) u_p (\Phi_{j(i)}(\alpha_{sp}\inv(a))- \alpha_{sp}\inv(a))u_p^* \alpha_t\inv(\eta_i(tp))\xi}{\eta} \right| \\
		&\leq \sum_{p \in G} \left\| u_p \big( \Phi_{j(i)} \big( \alpha_{sp}\inv(a) \big)- \alpha_{sp}\inv(a) \big) u_p^* \alpha_t\inv \big( \eta_i(tp) \big) \xi \right\| \left\| \alpha_s\inv \big( \xi_i(sp) \big) \eta \right\| \\
		&\leq \epsilon \sum_{p\in G} \left\| \alpha_t\inv \big( \eta_i(tp) \big) \xi \right\| \left\| \alpha_s\inv \big( \xi_i(sp) \big) \eta \right\| \\
		&\leq \epsilon \left( \sum_{p\in G} \left\| \alpha_t\inv \big( \eta_i(tp) \big) \xi \right\|^2 \right)^{1/2} \left( \sum_{p\in G} \left\| \alpha_s\inv \big( \xi_i(sp) \big) \eta \right\|^2 \right)^{1/2} \\
		&\leq \epsilon \left( \ip{\sum_{p\in G} \alpha_t\inv \big( \eta_i(tp)^*\eta_i(tp) \big) \xi}{\xi} \right)^{1/2} \left( \ip{\sum_{p\in G} \alpha_s\inv \big( \xi_i(sp)^*\xi_i(sp) \big)\eta}{\eta} \right)^{1/2} 
\end{split}
\]
By the first condition for the approximation property, it follows that 
\[
	\left\Vert \sum_{p \in G} \alpha_s\inv \big( \xi_i(sp)^* \big) u_p \big( \Phi_{j(i)}(\alpha_{sp}\inv(a) ) - \alpha_{sp}\inv(a) \big)u_p^* \alpha_t\inv \big( \eta_i(tp) \big) \right\Vert \leq C\epsilon.
\]
for some constant $C$.
Hence $\nrm[\varphi_{i,j(i)}(s,t)(\alpha_s\inv(a)) - \alpha_s\inv(a)] \to 0$ uniformly on $\Delta_R$ for all $a \in A$.
\end{proof}


\begin{remark}\label{re:Ozawaexactness}
{\rm 
Recall that a discrete group $G$ is called \emph{exact} if $\rgCst$ is an exact $C^*$-algebra.
Recall also Ozawa's characterisation of exact discrete groups~\cite{Oza00}: $G$ is exact if and only if the uniform Roe algebra $C^*_u(G) = \rcros{\ell^\infty(G)}{G}{\beta}$ is nuclear ($\beta$ denotes the translation action). 
We remark that this result can be seen through approximations of Herz--Schur multipliers of the dynamical system  $(\ell^\infty(G),G,\beta)$. 
In fact, if $\rgCst$ is an exact $C^*$-algebra, then by \cite{Oza00}*{Theorem 3} or by Theorem~\ref{th:exactcrossediffexactsystem} and Remark~\ref{re:exactSchmultspos} there exists a sequence $(\varphi_k)_k$ of positive (classical) Schur multipliers which are supported on strips and converge to 1 uniformly on strips.  
The positivity of $\varphi_k$ allows to find a net of functions $\xi_k:G\to\ell^\infty(G)^+$ that implements the amenability of the action $\beta$ of $G$ on $\ell^\infty(G)$, so that functions $\tilde\varphi_k$ determined by 
\[
    \tilde \varphi_k(s,r^{-1}s)=\sum_{p\in G}\xi_k(p)(s)\xi_k(r^{-1}p)(r^{-1}s)
\]
give a net with the same properties, see \textit{e.g.}\ \cite{AD02}*{Propositions 2.5, 3.5} or \cite{BrO08}*{Theorem 4.4.3}. 
It is known that $\ell^\infty(G)$ is nuclear (see \cite{BrO08}*{Proposition 2.4.2}).
Let $(\Phi_i)_{i\in I}$ be the approximating net of unital completely positive maps for $\ell^\infty(G)$ and define	$F_{k,j} : G \to \CB[\ell^\infty(G)]$ by 
\[
	 F_{k,j}(t)(a) \defeq \sum_{p\in G} \xi_k(p) \beta_p \big( \Phi_j(\beta_{p}^{-1}(a))\beta_t(\xi_k(t^{-1}p) \big) , \quad t \in G, \ a\in\ell^\infty(G) .
\] 
The arguments in \cite{MSTT18}*{Corollary 4.6} give that a subnet of  $F_{k,j}$, which we denote in the same way,  yields  Herz--Schur $(\ell^\infty(G) ,G,\beta)$-multipliers that satisfy \cite{MSTT18}*{Definition 4.1} and therefore implement nuclearity of $\rcros{\ell^\infty(G)}{G}{\beta}$.  
The other implication follows from the fact that any $C^*$-subalgebra of a nuclear $C^*$-algebra is exact. 
We note that as each $F_{k,j}$ is a finitely supported Herz--Schur $(\ell^\infty(G), G,\beta)$-multiplier of positive type (see \cite{MSTT18}), we have $F_{k,j}\in P(\ell^\infty(G),G,\beta)$, the positive definite elements of the Fourier--Stieltjes algebra of $(\ell^\infty(G),G,\beta)$ as defined in \cite{BC16}, see discussion after \cite{MSTT18}*{Remarks 2.10} and \cite{BC16}*{Theorem 4.5}. 
We obtain the following characterisation of exactness of $G$: $G$ is exact if and only if there exists a bounded net ${F_i}\in P(\ell^\infty(G),G,\beta)$ of finite support such that each $F_i(s)$ is a finite rank map and $\|F_i(s)(a)-a\|\to 0$ for any $s\in G$ and $a\in\ell^\infty(G)$.
}
\end{remark}


One can view a discrete group as a coarse space, in the sense of coarse geometry --- see Roe~\cite{Roe03}. 
An important and well-studied property of coarse spaces is Yu's \emph{property (A)}, which may be viewed as the analogue for coarse spaces of amenability of discrete groups; we refer to the survey \cite{Wil09} for the background on property (A), including a number of characterisations and the connection to $C^*$-algebra theory. 
The conditions: (i) $\rgCst$ is exact, (ii) $\rcros{\ell^\infty(G)}{G}{\beta}$ is nuclear, which appear in Remark~\ref{re:Ozawaexactness}, and (iii) the existence of positive Schur multipliers $\varphi_i$ supported on strips and converging to 1 uniformly on strips, are known to be equivalent to property (A) of $G$, where $G$ is viewed as a coarse space \cite{Wil09}*{Theorem 4.3.9}.
Since $\rgCst$ is a special case of the reduced crossed product we remark that Definition~\ref{de:exactsystem} (with positive Schur $A$-multipliers) may be regarded as a generalisation of property (A) to $C^*$-dynamical systems.

\section{The Operator Approximation Property}
\label{sec:OAP}

\noindent
Recall that a $C^*$-algebra $A$ is said to have the \emph{operator approximation property} (OAP) if there exists a net of finite rank continuous linear maps $(\Phi_i)_i$ on $A$ which converge to the identity in the stable point-norm topology, that is $\Phi_i \otimes \id_{\Cpt[\ell^2]} \to \id_{A \otimes \Cpt[\ell^2]}$ in point-norm, and $A$ is said to have the \emph{strong operator approximation property} (SOAP) if there exists a net of finite rank continuous linear maps $(\Phi_i)_i$ on $A$ which converge to the identity in the strong stable point-norm topology, that is $\Phi_i \otimes \id_{\Bd[\ell^2]} \to \id_{A \otimes \Bd[\ell^2]}$ in point-norm.
A locally compact group has the \emph{approximation property} (AP) if there is a net $(u_i)_i$ of finitely supported functions on $G$ which converge to the identity in the weak* topology of $\Mcb \falg$.
Haagerup and Kraus~\cite{HK94} proved that a discrete group $G$ has the AP if and only if $\rgCst$ has the SOAP if and only if $\rgCst$ has the OAP, and studied the behaviour of the AP and weak*OAP under the von~Neumann algebra crossed product.

Dynamical systems involving the action of a group with the AP have recently been studied by Crann--Neufang~\cite{CrN19} and Suzuki~\cite{Suz19}.
Suzuki shows that if a locally compact group $G$ has the AP then, for any $C^*$-dynamical system $(A,G,\alpha)$, $\rcrs$ has the (S)OAP if and only if $A$ has the (S)OAP.
In this section we give a Herz--Schur multiplier characterisation of the SOAP for $\rcrs$, generalising the result of Haagerup--Kraus.

The weak* topology of $\Mcb \falg$ comes from the introduction of the space $Q(G)$, which has $\Mcb \falg$ as its Banach space dual.
Recall that $Q(G)$ is defined as the completion of $C_c(G)$ in the norm given by the duality with $\Mcb \falg$:
\[
    \dualp{\phi}{u} \defeq \int_G \phi(s) u(s) ds, \quad \phi \in C_c(G),\ u \in \Mcb \falg .
\]
Following \cite{HK94}, for any concrete $C^*$-algebra $C$ there is a collection of functionals on $\CB[C]$: take $x \in C \otimes \Cpt[\ell^2]$ and $\psi \in (C'' \btens \Bd[\ell^2])_*$, or $x \in C \otimes \Bd[\ell^2]$ and $\psi \in (C \otimes \Bd[\ell^2])^*$, and define
\[
    \omega_{x , \psi} (T) \defeq \psi \big( (T \otimes \id) x \big) , \quad T \in \CB[C] .
\]
It is known \cite{HK94}*{Proposition 1.5} that
\[
    Q(G) = \{ \omega_{x , \psi} : x \in \rgCst \otimes \Cpt[\ell^2] ,\ \psi \in (\vN \btens \Bd[\ell^2])_* \} .
\]

By \cite{HK94}*{Proposition 1.4} we have $\omega_{x , \psi}\in Q(G)$ for any $x \in \rgCst \otimes \Bd[\ell^2]$  and $\psi \in (\rgCst \otimes \Bd[\ell^2])^*$. 
Moreover, $G$ has AP if and only if there is a net $(u_i)_i$ of finitely supported functions on $G$ such that $\omega_{x , \psi}(u_i)\to \omega_{x , \psi}(\id)$ for any such $x$ and $\psi$. 

\begin{definition}\label{de:QGforcrossedprods}
We define
\[
    Q(A,G,\alpha) \defeq \{ \omega_{x , \psi} : x \in \rcrs \otimes \Bd[\ell^2] ,\ \psi \in (\rcrs \otimes \Bd[\ell^2])^* \} .
\]
\end{definition}

We observe first that $Q(A,G,\alpha)$ is a linear space. 
In fact, let $\omega_{x_1,\psi_1}$, $\omega_{x_2,\psi_2}\in Q(A,G,\alpha)$. Let $\tilde \psi_i\in (M_2(\rcrs \otimes \Bd[\ell^2])^*$ given by $\tilde\psi_i(x)=\psi_i(x_{i,i})$, $x=(x_{i,j})_{i,j=1}^2\in M_2(\rcrs \otimes \Bd[\ell^2])$, $i=1,2$. 
We have 
\[
    \omega_{x_1,\psi_1}+\omega_{x_2,\psi_2}=\omega_{x_1\oplus x_2, \tilde\psi_1+\tilde\psi_2},
\]
with $x_1\oplus x_2\in M_2(\rcrs \otimes \Bd[\ell^2])\simeq \rcrs \otimes \Bd[\ell^2]$.

We have $Q(A,G,\alpha) \subseteq \HSm^*$, with the duality given by $\omega(F):=\omega(S_F)$ ($\omega\in  Q(A,G,\alpha) $, $F\in \HSm$), and there is an obvious contraction $\HSm \to Q(A,G,\alpha)^*$.

%
%

\begin{definition}\label{de:APforsystem}
Let $(A,G,\alpha)$ be a $C^*$-dynamical system, with $G$ a discrete group.
We say that $(A,G,\alpha)$ has the \emph{approximation property} (AP) if there is a net $(F_i)_{i \in I}$ of Herz--Schur $(A,G,\alpha)$-multipliers satisfying:
\begin{enumerate}[i.]
    \item each $F_i$ is finitely supported;
    \item for each $r \in G$  and $i\in I$, $F_i(r)$ is a finite rank map on $A$;
    \item $\omega(F_i) \to \omega(\id)$ for all $\omega \in Q(A,G,\alpha)$.
\end{enumerate}
\end{definition}

In the arguments below we will use the following standard consequence of the Hahn--Banach theorem several times: if $C$ is a convex set in a Banach space then the weak and norm closures of $C$ coincide.

\begin{remarks}\label{re:ondefofAP}
{\rm
\begin{enumerate}[i.]
    \item We cannot deduce the AP for $G$ from the AP for the system $(A,G,\alpha)$ in general. In fact, Lafforgue--de la Salle~\cite{LaSa11} have shown that there exist discrete exact groups without the AP, including $\mathrm{SL}(3,\ZZ)$, and if $G$ is such a group then (by \cite{Oza00}) $\rcros{\ell^\infty(G)}{G}{\beta}$ is nuclear and so has the SOAP. By Theorem~\ref{th:APequivtoOAPforsystem} below $(\ell^\infty(G), G,\beta)$ must then have the AP. 
    \item Condition (iii) above implies that $A$ has the SOAP. In fact, for $x\in A \otimes \Bd[\ell^2]$ and $\psi \in (\pi(A) \otimes \Bd[\ell^2])^*$,we have 
    \[
    	\begin{split}
    		\omega_{(\pi \otimes \id)(x) , \psi}(F_i(e)) &= \dualp{(\pi \otimes \id) \big( (F_i(e) \otimes \id)(x) \big)}{\psi} \\
    			&= \dualp{(S_{F_i} \otimes \id) \big( (\pi \otimes \id)(x) \big)}{\tilde \psi}\\
			& = \omega_{(\pi \otimes \id)(x), \tilde\psi}(S_{F_i}) \\
    			&\to \omega_{(\pi \otimes \id)(x)  ,\tilde \psi}(\id) \\
			&= \dualp{(\pi \otimes \id)(x)}{\tilde\psi} = \omega_{(\pi \otimes \id)(x) , \psi}(\id) .
    	\end{split}
    	\]
where $\tilde\psi$ is an extension of $\psi$ to a bounded linear functional on $(\rcrs) \otimes \Bd[\ell^2]$. We have therefore $F_i(e)$ converges in the strong stable point-weak topology to the identity map. 
By the Hahn--Banach theorem, there exists a net in the convex hull of $(F_i(e))_{i\in I}$ that converges to $\id$ in the strong stable point-norm topology. 
\end{enumerate}
}
\end{remarks}

\begin{theorem}\label{th:APequivtoOAPforsystem}
Let $(A,G,\alpha)$ be a $C^*$-dynamical system, with $G$ a discrete group.
The following are equivalent:
\begin{enumerate}[i.]
    \item $(A,G,\alpha)$ has the AP;
    \item $\rcrs$ has the SOAP.
\end{enumerate}
\end{theorem}
\begin{proof}
(i)$\implies$(ii) Conditions (i) and (ii) of Definition~\ref{de:APforsystem} imply that the maps $S_{F_i}$ on $\rcrs$ associated to the Herz--Schur $(A,G,\alpha)$-multipliers $F_i$ are finite rank.
By definition of $Q(A,G,\alpha)$ condition (iii) of Definition~\ref{de:APforsystem} implies that $(S_{F_i} \otimes \id_{\Bd[\ell^2]})x \to x$ weakly, for each $x \in (\rcrs) \otimes \Bd[\ell^2]$. 
Hence, by the Hahn--Banach theorem, there is a convex combination of the $S_{F_i}$ which converges to the identity in the strong stable point-norm topology. 

(ii)$\implies$(i) Let $(\Phi_i)_i$ be a net of completely bounded maps implementing the SOAP of $\rcrs$.
First, let us show that we may assume $\ran \Phi_i$ is contained in $B := \lspan \{ \pi(a) \lambda_r : a \in A ,\ r \in G \}$.
Write $\Phi_i = \sum_{j=1}^{k_i} \phi_j^i \otimes T_j^i$, where $\phi_j^i \in (\rcrs)^*$ and $T_j^i \in \rcrs$. If the $T_j^i$ are not elements of $B$ then, for each $n \in \NN$ find $T_{j,n}^i \in B$ such that $\Vert T_j^i - T_{j,n}^i \Vert < (n k_i \max_j \Vert \phi^i_j \Vert)\inv$.
Then $\Phi_{i,n} := \sum_{j=1}^{k_i} \phi_j^i \otimes T_{j,n}^i$ is a net which implements the SOAP of $\rcrs$ and has $\ran \Phi_{i,n} \subset B$.

Define
\[
    F_i : G \to \CB ;\ F_i(r)(a) \defeq \EE_A \big( \Phi_i( \pi(a) \lambda_r ) \lambda_{r\inv} \big) ,
\]
where $\EE_A$ is the canonical conditional expectation $\rcrs \to A$.
Then each $F_i$ is a Herz--Schur $(A,G,\alpha)$-multiplier by the same calculation as \cite{MSTT18}*{Proposition 3.4}, and the assumption on the range of $\Phi_i$ implies that $F_i$ is finitely supported and finite rank.

To prove the required convergence let $\Delta$ be the coaction of $G$ on $\rcrs$, so
\[
	\Delta (\pi(a) \lambda_r) = \lambda_r \otimes \pi(a) \lambda_r ,
\]
and define $V (\delta_r \otimes \xi) \defeq \delta_r \otimes \delta_r \otimes \xi$ ($r \in G,\ \xi \in \HH$); then $V^* \Delta(x) V = x$ ($x \in \rcrs$) and $S_{F_i} = V^* (\id \otimes \Phi_i) \Delta ( \cdot ) V$.
For $x \in (\rcrs) \otimes \Bd[\ell^2]$ and $\psi \in ((\rcrs) \otimes \Bd[\ell^2])^*$ calculate
\[
\begin{split}
	&| \omega_{x,\psi}(F_i) - \omega_{x,\psi}(\id) | = | \dualp{(S_{F_i} \otimes \id)(x) - x}{\psi} | \\
		&\qquad \qquad \qquad = | \dualp{(V \otimes \id)^*(\id \otimes \Phi_i \otimes \id)(\Delta \otimes \id)(x)(V \otimes \id) - x}{\psi} | \\
		&\qquad \qquad \qquad \leq \nrm[(\id \otimes \Phi_i \otimes \id)(\Delta \otimes \id)(x) - (\Delta \otimes \id)(x)] \nrm[\psi] \to 0 .
\end{split}
\]
\end{proof}

\begin{remark}\label{re:APbyCrannNeufang}
{\rm 
Crann--Neufang \cite{CrN19} also define a version of the AP for a $C^*$-dynamical system $(A,G,\alpha)$, based on a slice map property. 
They show \cite{CrN19}*{Corollary 4.11} that if $G$ has the AP then any $C^*$-dynamical system $(A,G,\alpha)$ has the AP in their sense. 
On the other hand, if $(A,G,\alpha)$ has the AP in the sense of Definition~\ref{de:APforsystem} then $A$ must have the SOAP.
Therefore these two notions of AP for a $C^*$-dynamical system are different.
}
\end{remark}

\begin{remark}
{\rm 
Finitely supported functions on $G$ are elements of the Fourier algebra $A(G)$, and hence the Fourier--Stieltjes algebra $B(G)$ of $G$. B\'{e}dos--Conti~\cite{BC16} introduced a notion of Fourier--Stieltjes algebra $B(\Sigma)$ of a $C^*$-dynamical system $\Sigma=(A,G,\alpha)$, by considering the set of $A$-valued coefficients of the equivariant representations of $\Sigma$.
It would be interesting to know whether the net $(F_i)_{i\in I}$  of Herz--Schur $(A,G,\alpha)$ multipliers in Definition~\ref{de:APforsystem} can be chosen from $B(\Sigma)$ in analogy with the group case. For this it would be enough to know whether $\Phi_i$ that implement the SOAP of $\rcrs$ can be chosen to be decomposable, see \cite{BC16}*{Proposition 4.1, Corollary 4.6}.  
} 
\end{remark}

It is known that if $(A,G,\alpha)$ is a $C^*$-dynamical system such that $G$ has the AP and $A$ has the SOAP then $\rcrs$ has SOAP (see \cite{Suz19}*{Theorem 3.6}). For discrete groups we now explain this result in our terms, further illustrating the utility of Herz--Schur multipliers for approximation results.
If $(u_i)_{i\in I}$ is a net of finitely supported functions in $A(G)$ such that $u_i\to 1$ weak$^*$ and $(\Phi_j)_{j\in J}$ is a net of finite rank linear maps on $A$ implementing SOAP of $A$, then we define  
\begin{equation}\label{hsm}
    F_{i,j}: G \to \CB[A] ;\ F_{i,j}(r)(a) := u_i(r) \Phi_j(a), \quad a\in A,\ i\in I,\ j\in J.
\end{equation}
We claim that there is a subnet of $(F_{i,j})_{i,j}$ which satisfies the conditions of Definition~\ref{de:APforsystem}. 
First we check that each $F_{i,j}$ is a Herz--Schur multiplier of $(A,G,\alpha)$.
For a finite sum $x = \sum_s \pi(a_s)\lambda_s \in \rcrs$ we have
\[
\begin{split}
    \|S_{F_{i,j}}(x)\| &= \left\| \sum_s u_i(s)\pi \big(\Phi_j(a_s) \big) \lambda_s \right\| \\
        &\leq \|u_i\|_\infty\|\Phi_j\|\sum_{s\in\supp u_i}\|a_s\|\leq \|u_i\|_\infty\|\Phi_j\|\sum_{s\in\supp u_i}\|\mathcal E_A(x\lambda_s^*)\| \\
        &\leq \|u_i\|_\infty\|\Phi_j\||\supp u_i|\|x\|,
\end{split}
\]
where $\mathcal E_A : \rcrs\to A$ is the conditional expectation onto $A$. 
This shows that $S_{F_{i,j}}$ is bounded. Similarly one sees that it is completely bounded.
Let now $a\in \rcrs \otimes \Bd \simeq \rcros{(A\otimes \Bd)}{G}{\alpha\otimes\id}$ and $f\in (\rcrs\otimes\Bd)^*$. 
Let $\mathcal{E}$ be the canonical conditional expectation  from $\rcros{(A\otimes \Bd)}{G}{\alpha\otimes\id}$ onto $A\otimes \Bd$ and write $a_s := \mathcal{E} (a\lambda_s^*)$ ($s\in G$).
Then
\[
\begin{split}
    & \big| \omega_{a,f}(F_{i,j})  -\omega_{a,f}(\id) \big| = \left| f \left( \sum_{s\in \supp u_i} u_i(s) \pi \big( (\Phi_j\otimes\id) (a_s) \big) \lambda_s \right) - f(a) \right| \\
        &\; \leq \left| f \left( \sum_{s\in \supp u_i} u_i(s) \big( \pi \big( (\Phi_j\otimes\id)(a_s) \big) - \pi(a_s) \big) \lambda_s \right) \right| + \big| \omega_{a,f}(u_i) - \omega_{a,f}(1) \big| .
\end{split}
\]
We have $\omega_{a,f}\in Q(G)$ (see \cite{Suz19}*{Lemma 3.2}) and hence the second term converges to zero. 
If $S=\supp u_i$ and $x_i(s)=u_i(s)a_s\in C(S, A\otimes \Bd ) \simeq A\otimes C(S,\Bd) \simeq A\otimes\Bd[\HH^{(n)}]$, where $n=|S|$, then, 
having the natural completely bounded embedding $\iota : C(S)\otimes A \to \rcrs$ we obtain 
\[
\begin{split}
    &\left| f \left( \sum_{s\in S} u_i(s) \big( \pi \big( \Phi_j\otimes\id(a_s) \big) - \pi(a_s) \big) \lambda_s \right) \right| \\
        &\qquad = \left| \big( f\circ (\iota\otimes\id) \big) \big( (\Phi_j\otimes\id)(x_i) - x_i \big) \right|.
    \end{split}
\]
As $\Phi_j \otimes \id_{\Bd[\HH^{(n)}]}\to \id_{A\otimes\Bd[\HH^{(n)}]}$ in point norm, for each $i$ we can choose $j(i)$ such that 
\[
    \left| \big( f\circ (\iota\otimes\id) \big) \big( \Phi_{j(i)} \otimes \id(x_i) - x_i \big) \right| \to_{i\in I} 0.
\]
Hence
\[
    \omega_{a,f}(F_{i,j(i)})-\omega_{a,f}(\id)\to_{i\in I} 0.
\]

We remark that the multipliers of the form (\ref{hsm}) are in B\'{e}dos--Conti's Fourier--Stieltjes algebra $B(\Sigma)$ if $\Phi_j$ are decomposable and commute with the action, \textit{i.e.}\ $\alpha_g\circ\Phi_j=\Phi_j\circ\alpha_g$ ($g\in G$).

\medskip

We finish this section with some observations concerning duality between the space of Herz--Schur multipliers of a dynamical system $(A,G,\alpha)$ and the space $Q(A,G,\alpha)$. 
First we introduce a subclass of the Herz--Schur multipliers considered so far. Recall that for a $C^*$-algebra $A$ we write $Z(A)$ for the centre of $A$ and $M(A)$ for the multiplier algebra of $A$.

\begin{definition}\label{de:centralmults}
Let $(A,G,\alpha)$ be a $C^*$-dynamical system, with $G$ discrete. 
A Herz--Schur $(A,G,\alpha)$-multiplier $F : G \to \CB$ will be called \emph{central} if $F(r) \in Z(M(A))$ for all $r \in G$.
Write $\mathfrak{S}_{\rm cent}(A,G,\alpha)$ for the algebra of central Herz--Schur $(A,G,\alpha)$-multipliers.
\end{definition}

Central multipliers have been used by a number of authors, including Anantharaman-Delaroche (see \textit{e.g.}\ \cite{AD02}*{Proposition 6.4}) and Dong--Ruan (see \cite{DR12}*{Section 3}) to study crossed products.
If $u : G \to \CC$ and we define $F_u : G \to \CB;\ F_u(r)(a) \defeq u(r) a$ ($r \in G,a \in A$) then it is shown in \cite{MTT18}*{Proposition 4.1} that $F_u$ is a Herz--Schur $(A,G,\alpha)$-multiplier if and only if $u \in \Mcb \falg$ (\textit{i.e.}\ $u$ is a classical Herz--Schur multiplier of $G$), and it is obvious that $F_u$ is central in this case; hence $\mathfrak{S}_{\rm cent}(A,G,\alpha)$ contains the canonical copy of $\Mcb\falg$.

Consider now a $C^*$-dynamical system $(C_0(Z),G,\alpha)$, where $Z$ is a locally compact space and $G$ a discrete group. 
If $F : G \to C_b(Z)$ is an element of $\mathfrak{S}_{\rm cent}(C_0(Z),G,\alpha)$ then there is a function $\varphi : Z\times G \to \CC$ such that $F_\varphi(r)(a)(z) = \varphi(z,r)a(z)$.  
We write $\|\varphi\|_{\rm m}$ for $\|F_\varphi\|_{\rm m}$ and write simply $\varphi$ instead of $F_\varphi$. Such multipliers are exactly Herz--Schur multipliers of the corresponding transformation groupoid $\mathcal G = Z\times G$ --- see \cite{MTT18}*{Section 5} and \cite{renault}.  

Assume now that $Z$ is discrete.  As $|\varphi(z,s)|\leq \|\varphi\|_{\rm m}$ ($(z,s)\in Z\times G$) any finitely supported function $\psi : Z\times G \to \CC$ defines a bounded linear functional on $\mathfrak S_{\rm cent}(C_0(Z),G,\alpha)$ by
\[
    \dualp{\psi}{\varphi} := \sum_{z,s} \psi(z,s)\varphi(z,s), \quad \varphi\in \mathfrak S_{\rm cent}(C_0(Z),G,\alpha).
\]
Define $\Qcent(Z,G,\alpha)$ to be the completion of $C_c(Z\times G)$ in  the dual space $(\mathfrak S_{\rm cent}(C_0(Z),G,\alpha))^*$. 
The following results are similar to the corresponding results in the group case (see \cite{HK94}*{Section 1} and \cite{BrO08}*{Lemma D.7, Lemma D.8}) and their proofs are essentially the same, so we give only sketches. 

\begin{lemma}\label{qxga}
If $a\in (\rcros{C_0(Z)}{G}{\alpha}) \otimes\Bd[\ell^2]$ and $f\in ((\rcros{C_0(Z)}{G}{\alpha}) \otimes\Bd[\ell^2])^*$, \textit{i.e.}\ $\omega_{a,f}\in Q(C_0(Z),G,\alpha)$, or if $a \in (\rcros{C_0(Z)}{G}{\alpha}) \otimes \Cpt[\ell^2]$ and $f\in ((\rcros{C_0(Z)}{G}{\alpha})''\btens\Bd[\ell^2])_*$, then $\omega_{a,f}\in \Qcent(Z, G,\alpha)$. 
\end{lemma}
\begin{proof}
Let $F$ be a finite subset of $Z\times G$ and consider $a=\sum_{(z,s)\in F}\pi\rtimes\lambda(\delta_{(z,s)})\otimes b(z,s)$, where $b(z,s)\in \Bd[\ell^2]$ for the first case and $b(z,s)\in \Cpt[\ell^2]$  for the second case. Then
\[
    \omega_{a,f}(\varphi) = \sum_{(z,s)\in F}\varphi(z,s) f \big( \pi\rtimes\lambda(\delta_{(z,s)})\otimes b(z,s) \big) , \quad \varphi\in \mathfrak{S}_{\rm cent}(C_0(Z),G,\alpha)
\]
showing that $\omega_{a,f}=\psi$ for some $\psi \in C_c(Z\times G)\subset \Qcent(Z,G,\alpha)$.  
For general $a$ the claim follows from $|\omega_{a,f}(\varphi)|\leq \|a\|\|f\|\|\varphi\|_{\rm m}$, $\varphi\in \mathfrak{S}_{\rm cent}(C_0(Z),G,\alpha)$.
\end{proof}

The following duality result closely resembles the duality between $Q(G)$ and $\Mcb \falg$.

\begin{proposition}\label{pr:QdualisHScent}
Let $(C_0(Z),G,\alpha)$ be a $C^*$-dynamical system, with $Z$ and $G$ discrete. 
We have the duality $\Qcent(Z,G,\alpha)^* = \mathfrak{S}_{\rm cent}(C_0(Z),G,\alpha)$; moreover 
\[
\begin{split}
    \Qcent(Z,G,\alpha) = \{ \omega_{a,f}: &a\in (\rcros{C_0(Z)}{G}{\alpha}) \otimes \Cpt[\ell^2], \\ &f\in ((\rcros{C_0(Z)}{G}{\alpha})'' \otimes \Bd[\ell^2])_*\} .
\end{split}
\]
\end{proposition}
\begin{proof}
The first part is similar to the proof of \cite{BrO08}*{Lemma D.8}. 
By definition $\Qcent(Z,G,\alpha)\subset (\mathfrak{S}_{\rm cent}(C_0(Z),G,\alpha))^*$, so we have a natural contraction
$\mathfrak S_{\rm cent}(C_0(Z),G,\alpha) \to (\Qcent(Z,G,\alpha))^*$. 
That it is an isometry follows from
\[
\begin{split}
    \|\varphi\|_{\rm m} &= \sup\{|\omega_{a,f}(\varphi)|:  a \in (\rcros{C_0(Z)}{G}{\alpha}) \otimes \Cpt[\ell^2])_1, \\ &\qquad \qquad \qquad \qquad f \in (((\rcros{C_0(Z)}{G}{\alpha})'' \btens \Bd[\ell^2])_*)_1\} \\
        &\leq \sup \{ |h(\varphi)| : h\in \Qcent(Z,G,\alpha), \| h \| \leq 1\} = \|\varphi\|_{(\Qcent(Z,G,\alpha))^*},
\end{split}
\]
where $B_1$ denotes the unit ball of a Banach space $B$. 

To see the surjection, to each $\psi\in \Qcent(Z,G,\alpha)^*$ we associate $\varphi(z,s) := \dualp{\psi}{\delta_{(z,s)}}$, observing that $\delta_{(z,s)}\in C_c(Z\times G)\subset \Qcent(Z,G,\alpha)$. 
For any $\omega\in C_c(Z\times G)$ we have 
\[
    \sum_{z,s} \varphi(z,s)\omega(z,s) = \dualp{\psi}{\omega} ,
\]
which implies that $\varphi$ is a central Herz--Schur multiplier of $(C_0(Z),G,\alpha)$ and moreover $\varphi = \psi$ on $\Qcent(Z,G,\alpha)$. 

The second statement follows from  Lemma~\ref{qxga} and  \cite{HK94}*{Lemma 1.6} with  $A = \rcros{C_0(Z)}{G}{\alpha}$, $X = \mathfrak{S}_{\rm cent}(C_0(Z),G,\alpha) \subset \CB[A, A'']$ and $E = \Qcent(Z,G,\alpha)$.
\end{proof}

Let $\theta$ be a faithful representation of $C_0(Z)$ on $\HH_\theta$ and form the corresponding regular covariant representation $\tilde{\theta} \rtimes \lambda$ of $(C_0(Z),G,\alpha)$ on $\ell^2(G,\HH_\theta)$. 
We write $\mathfrak{S}^\theta(C_0(Z),G,\alpha)$ for the set of Herz--Schur $(C_0(Z),G,\alpha)$-multipliers which admit a weak* continuous extension to $\overline{\tilde{\theta} \rtimes \lambda(\ell^1(G,C_0(Z)))}^{w^*}$ (this is the weak* closure of $\rcros{C_0(Z)}{G}{\alpha}$ in its representation on $\Bd[\ell^2(G,\HH_\theta)]$).

\begin{questions}
{\rm 
\begin{enumerate}
    \item Is $\mathfrak{S}_{\rm cent}(C_0(Z),G,\alpha)$ a dual space for general locally compact spaces $Z$? Is there a faithful representation $\theta$ of $C_0(Z)$ such that $\mathfrak{S}_{\rm cent}(C_0(Z),G,\alpha) \cap \mathfrak{S}^\theta(C_0(Z),G,\alpha)$ is a dual space?
    
    \item Is $\mathfrak{S}(C_0(Z),G,\alpha)$ a dual space? Is there a faithful representation $\theta$ of $C_0(Z)$ such that $\mathfrak{S}^\theta(C_0(Z),G,\alpha)$ is a dual space?

    \item For discrete $Z$ Proposition~\ref{pr:QdualisHScent} shows that $Q(C_0(Z),G,\alpha)$ is a subspace of the space 
\[
    \{ \omega_{a,f} : a\in (\rcros{C_0(Z)}{G}{\alpha}) \otimes \Cpt[\ell^2], f\in (\rcros{C_0(Z)}{G}{\alpha})''\btens \Bd[\ell^2])_* \} 
\] when both are considered as subspaces of $(\mathfrak{S}_{\rm cent}(C_0(Z),G,\alpha))^*$.
Is the same true for $C^*$-dynamical systems $(A,G,\alpha)$, with $G$ discrete, $A$ not assumed commutative and both subspaces considered as subspaces of $(\mathfrak{S}(C_0(Z),G,\alpha))^*$? In this case the result of Theorem 4.4 will be true when in Definition 4.2 (iii) we take  $\omega$ from  this larger space.
\end{enumerate}
}
\end{questions}

\noindent
\textbf{Acknowledgement.} We would like to thank the anonymous referee for valuable suggestions that led to improvements in the paper.

\bibliographystyle{}
\bibliography{exactnessbib}

\end{document}